\documentclass[12pt]{amsart}
\usepackage{amscd,epsfig,fullpage,url}
\usepackage[colorlinks=true, pdfstartview=FitV, linkcolor=blue, citecolor=blue, urlcolor=blue]{hyperref}
\usepackage{amssymb}
\usepackage{comment}
\usepackage{graphicx}
\usepackage{tikz-cd}

\usepackage[normalem]{ulem}

\DeclareMathOperator{\Ve}{Vec}

\DeclareMathOperator{\ad}{ad}

\DeclareMathOperator{\GL}{GL}
\DeclareMathOperator{\SL}{SL}

\DeclareMathOperator{\im}{im}

\DeclareMathOperator{\Sp}{Sp}

\DeclareMathOperator{\LND}{LNV}
\DeclareMathOperator{\LFD}{LFV}

\DeclareMathOperator{\diag}{diag}
\DeclareMathOperator{\Span}{span}
\DeclareMathOperator{\rk}{rk}

\begin{document}

\title[Bracket width of current Lie algebras]{Bracket width of current Lie algebras}
\author{Boris Kunyavskii, Ievgen Makedonskyi  and  Andriy Regeta}


\address{\noindent Department of Mathematics, Bar-Ilan University, \newline \indent Ramat Gan, Israel}
\email{kunyav@gmail.com }

\address{\noindent
Yanqi Lake Beijing Institute of Mathematical Sciences and Applications (BIMSA),
\newline
\indent Beijing, China
}
\email{makedonskyi.e@gmail.com}

\address{\noindent Institut f\"{u}r Mathematik, Friedrich-Schiller-Universit\"{a}t Jena, \newline
\indent   Germany}
\email{andriyregeta@gmail.com}

\begin{abstract}
The length of an element $z$ of a Lie algebra $L$ is defined as the
smallest number $s$ needed to represent $z$ as a sum
of $s$ brackets. The bracket width of $L$ is defined as
supremum of the lengths of its elements. Given a finite-dimensional simple Lie algebra $\mathfrak g$ over an algebraically closed field
$\Bbbk$ of characteristic zero, we study the bracket width
of current Lie algebras $L=\mathfrak g\otimes A$. We show that for
an arbitrary $A$ the width is at most 2. For $A=\Bbbk[[t]]$ and
$A=\Bbbk[t]$ we compute the width for algebras of types A and C.
\end{abstract}



\thanks{Research of the first author was supported by the ISF grant 1994/20. The third author is supported by DFG, project number 509752046.
Part of this research was accomplished when the authors were visiting the Max-Planck-Institut for Mathematics (Bonn) and the second author was visiting Bar-Ilan University. Support of these institutions is gratefully acknowledged.}

\maketitle

\newcommand{\adrien}[1]{\textcolor{purple}{[#1]}}

\newtheorem{claim}{Claim}
\newtheorem{question}{Question}
\newtheorem{theorem}{Theorem}
\newtheorem{corollary}[theorem]{Corollary}
\newtheorem{lemma}[theorem]{Lemma}
\newtheorem{proposition}[theorem]{Proposition}
\newtheorem{definition and proposition}{Definition and proposition}
\newtheorem{conjecture}[theorem]{Conjecture}
\newtheorem{st}{Step}
\setcounter{page}{1}

\newtheorem*{hyp}{Hypothesis (J)}
\newtheorem*{mthm}{Main Theorem}
\newtheorem*{thA}{Theorem A}
\newtheorem*{thB}{Theorem B}
\theoremstyle{definition}
\newtheorem{definition}[theorem]{Definition}
\newtheorem{remark}[theorem]{Remark}
\newtheorem{example}{Example}
\newtheorem{quest}{Question}
\newtheorem{problem}{Problem}

\def\bullitem{\medskip\item[$\bullet$]}
\newcommand{\name}[1]{\textsc{#1\/}}
\newcommand{\Rep}{\mbox{Rep}}
\newcommand{\ann}{\operatorname{Ann}}
\newcommand{\spec}{\operatorname{Spec}}
\newcommand{\syz}{\operatorname{Syz}}
\newcommand{\Syz}{\operatorname{Syz}}
\newcommand{\Der}{\operatorname{Der}}
\renewcommand{\d}{{\partial}}

\newcommand{\dx}{\frac{\partial}{\partial x}}
\newcommand{\dy}{\frac{\partial}{\partial y}}
\newcommand{\dz}{\frac{\partial}{\partial z}}

\newcommand{\0}{_{(0)}}
\newcommand{\1}{_{(1)}}
\newcommand{\2}{_{(2)}}
\newcommand{\3}{_{(3)}}
\newcommand{\4}{_{(4)}}

\newcommand{\gr}{\operatorname{gr}}

\def\Dp{\mathrm{D}_p}
\def\Dh{\mathrm{D}_h}
\def\Dz(z-1){\mathrm{D}_{z(z-1)}}
\def\Dq{\mathrm{D}_q}
\def\C{\Bbbk}
\def\N{\mathbb{N}}
\def\Z{\mathbb{Z}}
\newcommand{\lielnd}[1]{\langle \LND(#1) \rangle}
\newcommand{\lielfd}[1]{\langle \LFD(#1) \rangle}
\def\aut{\mathrm{Aut}}
\def\saut{\mathrm{SAut}}
\def\O{{\mathcal O}}
\def\isorightarrow {\xrightarrow{\sim}}

\def \itt #1,#2:{\medskip\item[$\bullet$]
     page\ \ignorespaces#1, line\ \ignorespaces#2:\ \ignorespaces}

\def\into{{\hookrightarrow}}

\section{Introduction}
Given a Lie  algebra $L$ over an infinite field $\Bbbk$, we define its bracket width  as the supremum of lengths $\ell (z)$, where $z$ runs over the derived algebra $[L,L]$ and $\ell (z)$ is defined as the smallest number $n$ of Lie brackets $[x_i,y_i]$ needed to
represent $z$ in the form
$
z=\sum_{i=1}^n[x_i,y_i].
$

There are many examples of Lie algebras of bracket width strictly bigger than one, see, e.g., \cite{Rom}. However, the width of any finite-dimensional complex {\it simple} Lie algebras is equal to one \cite{Br}.
For finite-dimensional simple {\it real} Lie algebras the problem of existence of an algebra of width greater than one is still wide open,
see \cite{Akh}.

The first examples of simple Lie algebras of bracket width greater than one were found only recently in \cite[Theorem A]{DKR21} among complex
{\it infinite-dimensional} algebras. Namely, they appeared among Lie algebras of vector fields $\Ve(C)$ on smooth affine curves $C$ with trivial tangent bundle, which are simple by \cite{Jor} and \cite[Proposition~1]{Sie96}. More recently, it was proved in \cite{MR} that the bracket width of such Lie algebras is less than or equal to three, and if in addition $C$ is a plane curve with the unique  place at infinity, the bracket width of $\Ve(C)$ equals two.

In the present paper, we  study the  bracket width of another class of infinite-dimensional Lie algebras, namely current Lie algebras.

Let $\Bbbk$ be an algebraically closed field of characteristic zero, $\mathfrak{g}$ be a finite-dimensional simple Lie $\Bbbk$-algebra,
$A$ be a commutative associative $\Bbbk$-algebra with the identity. The current  algebra corresponding to $\mathfrak g$ and $A$ is defined as the tensor product $\mathfrak{g}\otimes_{\Bbbk} A$ with the bracket
\[
[x\otimes a, y\otimes b]:=[x,y]\otimes ab.
\]
With respect to this bracket $\mathfrak g\otimes_{\Bbbk} A$ is a Lie algebra.

Our first result provides an upper estimate for the bracket width
of an arbitrary current algebra.

\begin{theorem} \label{upper}
The bracket width of $\mathfrak{g}\otimes_{\Bbbk} A$ is less than or equal to $2$.
\end{theorem}

The main object of our interest is the Lie algebra $\mathfrak g\otimes_{\Bbbk} A$ where
$A=\Bbbk[[t]]$ is the algebra of formal power series. In this case we expect
a more precise statement.

\begin{conjecture} \label{conj:lower}
Let $\mathfrak g$ be a finite-dimensional simple Lie algebra. Then the bracket width of $\mathfrak g\otimes_{\Bbbk} \Bbbk [[t]]$ is equal to $2$ if $\mathfrak g$ is of type $\mathrm A_n$ or $\mathrm C_n$ $(n\ge 2)$
and to $1$ otherwise.
\end{conjecture}

Our results partially confirm this expectation.

\begin{theorem} \label{lower}
\begin{itemize}
\item[{}]
\item[(i)] The bracket width of $\mathfrak{sl}_2\otimes \Bbbk[[t]]$ is equal to $1$.
\item[(ii)] If $\mathfrak g$ is of type $\mathrm A_n$ or $\mathrm C_n$    $(n\ge 2)$, the bracket width of $\mathfrak g\otimes \Bbbk[[t]]$ is equal to $2$.
\end{itemize}
\end{theorem}
Some arguments supporting the conjecture for the types other than
$\mathrm A_n$ or $\mathrm C_n$ will be given later, in Section
\ref{sec:concl}.

We deduce from (the proof of) Theorem \ref{lower} some results
on other current algebras.

\begin{corollary} \label{cor:kt}
Let $\mathfrak g=\mathfrak{sl}_n$ or $\mathfrak{sp}_{2n}$ $(n\ge 2)$.
Then for $A=\Bbbk[t]$ the width of $\mathfrak g\otimes_{\Bbbk}A$ is equal to 2.
\end{corollary}

This statement can be generalized to a wider class of rings $A$ as follows.


\begin{corollary} \label{cor:irred}
Let $\mathfrak g=\mathfrak{sl}_n$ or $\mathfrak{sp}_{2n}$ $(n\ge 2)$.
Let $A$ be a ring containing an ideal $\mathfrak a$
such that the quotient $\bar A=A/\mathfrak a$ is
a two-dimensional $\Bbbk$-algebra.
Then the width of $\mathfrak g\otimes_{\Bbbk}A$ is equal to 2.

\end{corollary}

\section{Proofs} 

We begin with the following general statement on finite-dimensional
simple Lie algebras \cite[Theorem 26]{BN11}.

\begin{proposition}\label{commutatorsregular} Let $\mathfrak{g}$ be a simple finite-dimensional Lie algebra defined over an arbitrary infinite field of characteristic not $2$ or $3$. Then
   there exist $w_1, w_2 \in \mathfrak{g}$ such that  \[\mathfrak{g}=[w_1,\mathfrak{g}]+[w_2,\mathfrak{g}].\]
\end{proposition}

This immediately implies Theorem \ref{upper}.

\begin{proof}[Proof of Theorem \ref{upper}]

 Consider a linear basis of $A$, $A=\langle 1=a_0,a_1,a_2,\dots \rangle$.
We have
\[\mathfrak{g}\otimes A=\mathfrak{g}\otimes 1\oplus\mathfrak{g}\otimes a_1\oplus\mathfrak{g}\otimes a_2\oplus\dots\]

 Any element $z$ of  $\mathfrak{g}\otimes A$ can be written in the form $z=\sum_{i=0}^k z_i\otimes \alpha_ia_i$ with $z_i \in \mathfrak{g}$, $\alpha_i \in \Bbbk$. By Proposition \ref{commutatorsregular}, for every $z_i$
 there exist $x_i,y_i \in \mathfrak g$ such that
\[z_i=[w_1,x_i]+[w_2,y_i].\]

Thus, we have:
\[z=\sum_{i=0}^k z_i\otimes \alpha_ia_i=\left[w_1, \sum_{i=0}^k x_i\otimes \alpha_ia_i \right]+\left[w_2, \sum_{i=0}^k y_i\otimes  \alpha_ia_i \right].\]
This completes the proof.
\end{proof}



\begin{proof}[Proof of Theorem \ref{lower}]
Our first step consists in reformulating the property of $L:=\mathfrak g\otimes \Bbbk[[t]]$ to be of bracket width one as some condition on $\mathfrak g$.

\begin{proposition} \label{prop:crit}

\begin{enumerate}
\item[{}]
\item[(i)] The bracket width of $L$ is equal to $1$ if and only if $\mathfrak g$ satisfies the following condition~$(*)$:
every nonzero element $c\in \mathfrak g$ can be
represented as a bracket of elements without common centralizer, i.e. there exist $a,b\in \mathfrak g$ such that $c=[a,b]$
and
\begin{equation} \label{eq:central}
C_{\mathfrak g}(a)\cap C_{\mathfrak g}(b)=(0).
\end{equation}
\item[(ii)] Assume that $A$ 
satisfies the conditions of Corollary \ref{cor:irred}. Then condition $(*)$ is necessary for $\mathfrak g \otimes A$ to be of bracket width $1$.
\end{enumerate}
\end{proposition}

The following  simple lemma is needed for the proof of Proposition \ref{prop:crit}.
\begin{lemma} \label{lem:reform}
Condition \eqref{eq:central} is equivalent to the following one:
\begin{equation} \label{eq:image}
\im(\ad a) + \im (\ad b) = \mathfrak g.
\end{equation}
\end{lemma}

\noindent{\it {Proof of Lemma \ref{lem:reform}.}}
Let $(,)$ denote the Killing form on $\mathfrak g$, and let
$V \subset \mathfrak g$ denote the orthogonal complement to $\im(\ad a) + \im (\ad b)$.

Suppose that condition \eqref{eq:central} holds and prove \eqref{eq:image}.
Assume to the contrary that \eqref{eq:image} does not hold, i.e. $V\ne (0).$
Let $d$ be a nonzero element of $V$. Then for any $e\in \mathfrak g$ we have
$([e,a],d)=([e,b],d)=0$. As the Killing form is invariant, this gives
$(e,[a,d])=(e,[b,d])=0.$ Since $e$ is an arbitrary element of $\mathfrak g$
and the Killing form is non-degenerate, we have $[a,d]=[b,d]=0$, i.e.
$d$ centralizes both $a$ and $b$, contradiction.

Conversely, suppose that condition \eqref{eq:image} holds and prove \eqref{eq:central}. Assume to the contrary that $C_{\mathfrak g}(a)\cap C_{\mathfrak g}(b)\ne (0).$ Let $d\ne 0$ centralize both $a$ and $b$. Then
the same argument as above shows that $d\in V$, contradiction. \qed

\begin{proof}[Proof of  Proposition \ref{prop:crit}]

(i) Suppose that $\mathfrak g$ satisfies condition $(*)$ and show that the bracket width of $L=\mathfrak g\otimes \Bbbk[[t]]$ is equal to 1.
Let
\[
z=z_0+z_1\otimes t+z_2\otimes t^2+\dots, \quad z_i\in \mathfrak g,
\]
be an arbitrary element of $L$. We want to represent it as $z=[x,y]$ where
\[
x=x_0+x_1\otimes t+x_2\otimes t^2+\dots, \quad y=y_0+y_1\otimes t+y_2\otimes t^2+\dots,\quad x_i, y_i\in \mathfrak g,
\]
which gives the equation
\[
\sum_{k=0}^{\infty}\sum_{i+j=k} [x_i,y_j]\otimes t^k = \sum_{k=0}^{\infty} z_k \otimes t^k,
\]
which, in turn, yields the system of equations
\begin{equation} \label{eq:bracket2}
\begin{aligned}
\relax
[x_0,y_0] &= z_0 \\
[x_0,y_1]+[x_1,y_0] &= z_1 \\
&\dots \\
[x_0,y_k]+[x_k,y_0] &= z_k - \sum_{i=1}^{k-1}[x_i,y_{k-i}] \\
&\dots
\end{aligned}
\end{equation}
Without loss of generality we can assume $z_0\ne 0$.
Condition $(*)$ together with Lemma \ref{lem:reform} allows one to find $x_0$, $y_0$ and then $x_1$, $y_1$. By induction,
we find all other $x_k$ and $y_k$.

Conversely, assuming that the bracket width of $L$ equals $1$, looking at the zeroth and first equations of the above system and
applying Lemma \ref{lem:reform} once again, we conclude that condition $(*)$ holds in $\mathfrak g$.

(ii) Suppose that $A$ satisfies the conditions of Corollary
\ref{cor:irred} and that the width of $\mathfrak g\otimes A$ is equal to $1$.
We have to show that condition $(*)$ holds. We argue as in the
necessity part of the proof of (i). Namely, let $\{1,\bar t\}$
be a linear basis of $\bar A=A/\mathfrak a$, and fix a preimage $t$ of
$\bar t$. Let
$z=z_0\otimes 1+z_1\otimes t$ be an element of  $\mathfrak g\otimes A$
with $z_0\ne 0$. Any such $z$ can be represented in the form $z=[x,y]$
with
$$
x=x_0+x_1\otimes t+\sum_{i\ge 2} x_i\otimes a_i, \quad y=y_0+y_1\otimes t +
\sum_{i\ge 2} y_i\otimes b_i
$$
with $a_i, b_i\in \mathfrak a$. We then arrive at the system consisting
of the first two equations in \eqref{eq:bracket2}. By Lemma
\ref{lem:reform}, condition $(*)$ holds in $\mathfrak g$.
\end{proof}

We now continue the proof of Theorem \ref{lower} using the criterion obtained in Proposition \ref{prop:crit}.

\medskip

\noindent{\it {Proof of Theorem \ref{lower}}} (i). This case is easy because any element $c$ of $\mathfrak g=\mathfrak{sl}_2$ is either nilpotent or semisimple.

First assume that $c$ is nilpotent. We can use the natural representation of $\mathfrak g$ and write $c=e=\left(\begin{matrix} 0 & 1 \\ 0 & 0\end{matrix}\right).$
Take $a=h/2=\diag(1/2,-1/2)$, $b=c$. We obtain $[a,b]=c$, $C_{\mathfrak g}(a) =\Span(a)$, $C_{\mathfrak g}(b) =\Span(b)$, so that
$C_{\mathfrak g}(a)\cap C_{\mathfrak g}(b) =(0)$.

Let now $c$ be semisimple, write $c=h=\left(\begin{matrix} 1 & 0 \\ 0 & -1\end{matrix}\right).$
Take $a=e = \left(\begin{matrix} 0 & 1 \\ 0 & 0\end{matrix}\right)$, $b=\left(\begin{matrix} 0 & 0 \\ 1 & 0\end{matrix}\right).$
We obtain $[a,b]=c$, $C_{\mathfrak g}(a) =\Span(a)$, $C_{\mathfrak g}(b) =\Span(b)$, so that
$C_{\mathfrak g}(a)\cap C_{\mathfrak g}(b) =(0)$, as above. Condition $(*)$ is satisfied,
hence the bracket width of $L$ is equal to $1$, as claimed.

\medskip

\noindent{\it {Proof of Theorem \ref{lower}}} (ii). We have to prove that for $\mathfrak g=\mathfrak{sl}_n$ ($n\ge 3$) and $\mathfrak g=\mathfrak{sp}_{2n}$
($n\ge 2$)
the bracket width of $L$ is greater than 1. Together with the upper estimate from Theorem \ref{upper} this will imply that
the width is equal to $2$.

We thus have to prove that $\mathfrak g$ does not satisfy condition $(*)$. This means that we have to exhibit
an element $c$ such that any $a,b$ with $[a,b]=c$ have a nonzero common centralizer. We shall choose $c$ to be
a rank 1 matrix in the natural representation of $\mathfrak g$.
This will allow us
to apply the following general lemma from linear algebra.

\begin{lemma} \label{lem:Gur} \cite{Gu}
Let $A,B$ be square matrices such that $\rk (AB-BA)\le 1$. Then one can simultaneously
conjugate $A$ and $B$ to upper triangular form.
\end{lemma}

\begin{remark}
See \cite[Lemma~12.7]{EG} for an alternative proof of Guralnick's lemma (attributed to Rudakov).
\end{remark}

We now go over to a more general set-up, using the notion of
{\it almost commuting scheme} of $\mathfrak g$, see \cite{GG}, \cite{Lo}. First, let us define it
for $\mathfrak g=\mathfrak{sl}_n$.

Let $R$ denote the vector space $\mathfrak {sl}_n^{\oplus 2}\oplus \mathbb C^n \oplus (\mathbb C^n)^*$. The subscheme $M_n\subset R$ is defined  as
\begin{equation} \label{comm-A}
\{ (x,y,i,j)\in \mathfrak {sl}_n^{\oplus 2}\oplus \mathbb C^n \oplus (\mathbb C^n)^* \mid  [x,y]+ij=0  \}
\end{equation}
and is called the almost commuting scheme of $\mathfrak{sl}_n$.

In a similar way, for $\mathfrak g=\mathfrak{sp}_{2n}$ we consider its
natural representation $\mathbb C^{2n}$, identify $S^2(\mathbb C^{2n})$
with $\mathfrak{sp}_{2n}$ and thus view $i^2\in S^2(\mathbb C^{2n})$
as an element of $\mathfrak{sp}_{2n}$. The almost commuting scheme $X_n$
of $\mathfrak{sp}_{2n}$ is then defined similarly to \eqref{comm-A}:
\begin{equation} \label{comm-C}
X_n := \{(x,y,i)\}\in \mathfrak {sp}_{2n}^{\oplus 2}\oplus \mathbb C^{2n} \ | \  [x,y]+i^2=0\}.
\end{equation}

Note that both varieties carry a natural action of $G=\SL_n$ or $\Sp_{2n}$,
respectively. Say, $G=\SL_n$ acts on $M_n$ by the formula
$$
g(x,y,i,j)=(gxg^{-1}, gyg^{-1}, gi, jg^{-1}).
$$
Note that such an action on $M_n$ is well-defined since $(gxg^{-1}, gyg^{-1}, gi, jg^{-1}) \in M_n$ whenever $(x,y,i,j) \in M_n$ as the following computations show:
\[
[gxg^{-1}, gyg^{-1}] + gijg^{-1} = g[x,y]g^{-1} + gijg^{-1} = g(-ij)g^{-1} + gijg^{-1} = 0.
\]

We have the following generalization of Guralnick's
lemma, see \cite[Lemma~12.7]{EG} and \cite[Lemma~2.1]{Lo}.

\begin{lemma} \label{lem:tri}
Let $(x,y,i,j)\in M_n$ $($resp. $(x,y,i)\in X_n)$. Then there is a Borel
subalgebra $\mathfrak b$ of $\mathfrak g$ that contains both $x$ and $y$. \qed
\end{lemma}


We continue the proof of Theorem \ref{lower}(ii). In our new notation, we have to prove that
given any $(a,b,i,j)\in M_n$ $($resp. $(a,b,i)\in X_n)$, the elements
$a$ and $b$ have nonzero common centralizer in $\mathfrak g$.

First suppose that both $a$ and $b$ are nilpotent. Then they are both contained
in the nilradical $\mathfrak{n}$ of $\mathfrak{b}$. Since $\mathfrak{n}$ is nilpotent,
its centre is nontrivial, and any its element is a common centralizer of $a$ and $b$.

So assume that at least one of $a$ and $b$ is not nilpotent and consider
the orbit $O=G(a,b,i,j)$ (resp. $G(a,b,i)$). For the sake of brevity,
in both cases we denote it by $O_{a,b}$.

In the sequel, we shall use $Q_{\mathfrak g}$ as a common notation for $M_n$ and $X_n$.
Lemma \ref{lem:tri} implies the following description of closed
orbits   in  $Q_{\mathfrak g}$.

\begin{lemma} \label{closed}
The orbit $O_{a,b}$ is closed if and only if $a$ and $b$ are commuting
semisimple elements.
\end{lemma}

\begin{proof}
Case A:
By Guralnick's Lemma, it is sufficient to consider the case of upper triangular matrices. But then the orbit of a one-parameter group of matrices (the one-parameter torus corresponding to the coweight $2\rho^\vee$)
\[
  \{  T_t := \diag(t^{n-1},t^{n-3},\dots,t^{-n+1}) \mid t \in \Bbbk^* \}
\] is a quasi-affine subvariety
\[
\{ (T_taT_t^{-1},T_tbT_t^{-1},T_ti,jT_t^{-1}) \mid t \in \Bbbk^* \} \subset Q_{\mathfrak g}
\]
which contains $(a_s,b_s,0,0)$ in its closure. This proves the statement  for $\mathfrak g = \mathfrak {sl}_n$.

Case C:
A similar argument works in the case of $\mathfrak  g =\mathfrak{sp}_{2n}$ (see also \cite[Corollary~2.2]{Lo}).
\end{proof}

By Lemmas \ref{lem:tri} and \ref{closed}, we may assume that
the closure of $O_{a,b}$ contains the closed orbit $O_{a_s,b_s}$ where
$a_s$ and $b_s$ are commuting diagonal matrices.

Following \cite[Section~2.2]{Lo},
denote by $\mathfrak l$ the common centralizer in $\mathfrak g$ of the (commuting) elements $a_s$ and $b_s$, it is a Levi subalgebra of
$\mathfrak g$. Denote by $L$ the corresponding Levi subgroup of $G$.

We are going to apply Luna's slice theorem \cite{Lu}. We will use the exposition of the slice method from lecture notes by Kraft \cite{Kr}.
So let $X=Q_{\mathfrak g}$, $O=O_{a_s,b_s}$, then the almost commuting scheme
$Q_{\mathfrak l}$ of $\mathfrak l$ is the required \'etale slice $S$ (see
\cite[Lemma~2.4]{Lo}), so that
in an \'etale neighbourhood of $O$ we have an excellent morphism
\begin{equation} \label{phi}
\varphi\colon G\times^{L}S\to X
\end{equation}
taking the image  $[g,s]\in G\times^LS$ of the pair $(g,s)\in G\times S$ to $gs$; in particular, $\varphi$ is \'etale,
and its image is affine and open in $X$, see \cite[Theorem~4.3.2]{Kr}.

Since $L$ is a {\it reductive} group of the form  $\prod_{i=1}^k\GL_{n_i}\times \Sp_{2n_0}$ where each of the first $k$ factors  necessarily has a nontrivial centre, $S=Q_{\mathfrak l}$ is of the form
\begin{equation} \label{slice}
\mathbb C^{2k}\times \prod_{i=1}^kM_{n_i} \times X_{n_0},
\end{equation}
where the $n_i$ correspond to the partition
$\rk L=n_0+n_1+\dots +n_k$ with $n_0\ge 0$, $n_i>0$ $(i=1,\dots ,k)$,
and $\mathbb C^{2k}$ is identified with $\mathfrak z(\mathfrak l)^{\oplus 2}$, see \cite[2.2]{Lo}.

The presence of the nontrivial $\mathfrak z(\mathfrak l)$ is of critical
importance: it guarantees the existence of a pair $(z,z')\in\mathfrak  z(\mathfrak l)^{\oplus 2}$ with nonzero components each of those   centralizes both $x_s$ and $y_s$ (of course, the simplest choice is
$z=x_s$, $z'=y_s$).

Thus for any element of
$$
Q_{\mathfrak l}\subset \oplus_{i=1}^k(\mathfrak{gl}_{n_i}^{\oplus 2}\oplus \mathbb C^{2n_i} \oplus
(\mathbb C^*)^{2n_i}) \oplus \mathfrak{sp}_{2n_0}^{\oplus 2} \oplus \mathbb C^{2n_0}
$$
of the form
$$(x_{n_1},y_{n_1},i_{n_1},j_{n_1}, \dots, x_{n_k},y_{n_k},i_{n_k},j_{n_k},
x_{n_0},y_{n_0},i_{n_0})
$$
the elements $x=(x_{n_1},\dots ,x_{n_k}, x_{n_0}), \,
y=(y_{n_1},\dots ,y_{n_k}, y_{n_0})\in \oplus_{i=1}^k\mathfrak{gl}_{n_i}\oplus \mathfrak{sp}_{2n_0}$ have a nonzero common centralizer.



Given any  finite-dimensional simple Lie algebra $\mathfrak g$,
denote by $F_{\mathfrak g}$ the set of pairs $(x,y)\in\mathfrak g^{\oplus 2}$
such that $x$ and $y$ have a nonzero common centralizer, and let  $U_{\mathfrak g}:=\mathfrak g^{\oplus  2}\setminus F_{\mathfrak g}$
denote its complement.

The following lemma is a variation on a theme of Arzhantsev \cite[Section~5]{Ar}.

\begin{lemma} \label{lem:Ar}
The set $U_{\mathfrak g}$ is open and Zariski dense in $\mathfrak g^{\oplus 2}$.
\end{lemma}

\begin{proof}
First fix a pair $(a,b)\in \mathfrak g\oplus \mathfrak g$ and define a linear map $T_{a,b}\colon \mathfrak g \rightarrow \mathfrak g \oplus \mathfrak g$
by
\[T_{a,b}(x)=([a,x],[b,x]).\]
Let now $V$ denote the vector space of linear maps $\mathfrak g\to
\mathfrak g\oplus \mathfrak g$. Define
$\psi\colon \mathfrak g \oplus \mathfrak g  \rightarrow V$ by
$\psi (a,b)=T_{a,b}$, it is 
a linear map. 
Let $W\subset V$ denote the set of maps of maximal rank, it is open in $V$.
Consider the preimage $\psi^{-1}(W)$. Note that
$$\ker T_{a,b}=C_{\mathfrak g}(a)\cap C_{\mathfrak g}(b).$$
Hence we have $\psi^{-1}(W)=U_{\mathfrak g}$ because if $a$ and $b$ have a non-zero common centralizer, then $\ker T_{a,b}\ne 0$
and therefore the rank of $T_{a,b}$ is strictly less than $\dim\mathfrak g$.  Thus $U_{\mathfrak g}$ is open in $\mathfrak g\oplus \mathfrak g$ as the preimage of an open set. It remains to note that
$U_{\mathfrak g}$ is non-empty. Indeed (see \cite[Remark~3]{Ar}),
any simple finite-dimensional Lie algebra $\mathfrak g$ is two-generated and centreless, so that any pair of generators $(a,b)$ belongs to $U_{\mathfrak g}$.
The lemma is proven.
\end{proof}

Let $F'_{\mathfrak g}:=F_{\mathfrak g}\oplus \mathbb C^{2n}$,
embed it into $\mathfrak g^{\oplus 2}\oplus \mathbb C^{2n}$ and define
$F''_{\mathfrak g}:=F'_{\mathfrak g}\cap Q_{\mathfrak g}$.

By Lemma \ref{lem:Ar}, $U_{\mathfrak g}$ is open and Zariski dense in $\mathfrak g^{\oplus  2}$.
Hence $U'_{\mathfrak g}:=U_{\mathfrak g}\oplus \mathbb C^{2n}$, embedded into
$\mathfrak g^{\oplus 2}\oplus \mathbb C^{2n}$, is also open.
Therefore $F'_{\mathfrak g}$ is closed in  $\mathfrak g^{\oplus  2}\oplus \mathbb C^{2n}$, and  thus $F''_{\mathfrak g}$
is closed in $Q_{\mathfrak g}$.

We wish to prove that $F''_{\mathfrak g}=C_{\mathfrak g}$. This will establish the statement (ii) of the theorem.

Assume to the contrary that there exists a quadruple $(x,y,i,j)\in M_n$ (resp. a triple $(x,y,i)\in X_n$)
such that $(x,y)\notin F_{\mathfrak g}$. Consider the morphism $\varphi$ defined in \eqref{phi}.
As said, its image is open. On the other hand, for all elements $(x,y,i,j)$ (resp. $(x,y,i)$) lying
in this image we have $(x,y)\in F''_{\mathfrak g}$ because the corresponding property to have
a nonzero common centralizer holds in $S=Q_{\mathfrak l}$, as said above. Since the image of $\varphi$
is open, the closure of $F''_{\mathfrak g}$ is $X=Q_{\mathfrak g}$, contradiction.
This proves the statement. 
\end{proof}

Corollary \ref{cor:irred} now follows from Theorem \ref{upper},
the proof of Theorem \ref{lower}(ii)
and Proposition \ref{prop:crit}(ii). Since $A=\Bbbk[t]$ satisfies
the conditions of Corollary \ref{cor:irred} with
$\mathfrak a=t^2\Bbbk[t]$, Corollary \ref{cor:kt} follows as well.

\section{Concluding remarks} \label{sec:concl}

We finish with some remarks on what was not done and what should
(and hopefully will) be done in the near future.

$\bullet$ The first tempting goal is to settle the remaining cases
of Conjecture \ref{conj:lower}. By Theorem \ref{upper}, the width
of $\mathfrak g\otimes\Bbbk[[t]]$ is at most $2$. To prove that
it is equal to $2$, we have to exhibit an element of $c\in\mathfrak g$
such that for every representation $c=[a,b]$ the elements $a$ and $b$
have a nonzero common centralizer, as in the proof of Theorem \ref{lower}(ii). There we took an element $c$ of the minimal nonzero nilpotent orbit $\mathbb O_{\textrm{min}}$ (it is well known that
there exists a unique such orbit \cite[Theorem~4.3.3 and Remark~4.3.4]{CM}) and used simultaneous triangularization of $a$ and $b$. However, this method breaks down for all types other than $\textrm A_n$ and $\textrm C_n$ as shown by Losev in \cite[Remark~2.3]{Lo}: in simple algebras of all those types there are elements
$c=[a,b]\in\mathbb O_{\textrm{min}}$ such that $a$ and $b$ do not lie in
a common Borel subalgebra. This gives a certain evidence that these
algebras are of width $1$.

$\bullet$ It would be interesting to look at other current algebras.
Say, by Theorem \ref{upper} it is known that the bracket width of $\mathfrak{sl}_2\otimes \Bbbk[t]$ is less than or equal to $2$. Although  we know that $\mathfrak{sl}_2\otimes \Bbbk[[t]]$ has bracket width  $1$, we still do not know the bracket width  of  $\mathfrak{sl}_2\otimes \Bbbk[t]$.

$\bullet$ In a similar vein, it would be interesting to compute the
width of the loop algebras $\mathfrak g\otimes\Bbbk[t,t^{-1}]$.

$\bullet$ Our final remark concerns the parallel results on the width
of the finite-dimensional Lie $R$-algebras $\mathfrak{gl}_n(R)$ obtained for various rings $R$ in a slightly different context. Namely, it is known that the width of such a Lie algebra is at most $2$. This was first proved by Amitsur and Rowen \cite{AR} for division rings $R$
and then generalized to arbitrary commutative rings \cite{Ros}
(and even to non-commutative rings \cite{Me}). This looks like an
almost full analogue of our Theorem \ref{upper}, modulo
the transition from $\mathfrak{gl}_n$ to $\mathfrak{sl}_n$, which may be
a non-trivial task, see \cite{St} (in the latter paper it is also
shown that the width of $\mathfrak{sl}_n(R)$ is equal to $1$ if $R$ is a principal ideal domain).

However, none of these
results implies the other: the bracket width of the
{\it infinite-dimensional} Lie $\Bbbk$-algebra ${\mathfrak{sl}}_n\otimes
_{\Bbbk}R$ is {\it a priori} unrelated to the bracket width of the
{\it finite-dimensional} $R$-algebra $\mathfrak{sl}_n(R)$. In light
of the existing parallels, it would be interesting to compute
the bracket width of the finite-dimensional simple Lie $R$-algebras
${\mathfrak g}(R)$.


\medskip

\noindent
{\emph{Acknowledgements.}}
We thank \name{Alexander Elashvili}, \name{Alexander Premet},
and \name{Oksana Yakimova} for helpful discussions.
Our special thanks go to \name{Ivan Losev} who suggested the idea of
applying the slice method to almost commuting varieties.


\end{document}